\documentclass[12pt,a4paper]{article}
\usepackage{graphicx} % Required for inserting images
\usepackage{todonotes}
\usepackage{a4wide}
\usepackage[english]{babel}
\usepackage{xcolor}
\usepackage{graphicx}
\usepackage{transparent}
\usepackage[default,scale=1]{opensans}
\usepackage[T1]{fontenc}
\usepackage[utf8]{inputenc}
\usepackage{enumerate}
\usepackage{comment}
\usepackage{amsmath,amsthm,amssymb,mathtools,cancel, stackengine, mathrsfs}
\usepackage{float}

\theoremstyle{plain}
\newtheorem{thm}{Theorem}[section]
\newtheorem*{thm*}{Theorem}

\newtheorem*{cor*}{Corollary}
\newtheorem{lem}[thm]{Lemma}

\newcommand{\C}{\mathbb{C}}

\newcommand{\R}{\mathbb{R}}

\newcommand{\beqa}{\begin{eqnarray*}}
	\newcommand{\eeqa}{\end{eqnarray*}}

\def\R{\right)}

\def\<{\left<}
\def\>{\right>}

\def\mv1{M_v^1}

\def\R{\mathbb{R}}
\def\Ren{\mathbb{R}^d}

\def\Sn2{S_{2}(L^{2}(\Ren))}
\def\S1{S_{1}(L^{2}(\Ren))}
\def\sig00{\sigma_{0,0}}

\newcommand{\D}{\mathcal{D}}

\usepackage{pifont}% http://ctan.org/pkg/pifont

\theoremstyle{definition}

\theoremstyle{remark}
\newtheorem{re}{Remark}

\newtheorem{example}{Example}

\title{Quantum Harmonic Analysis  and the Structure in Data: Augmentation}
\author{Monika D\"orfler \and Franz Luef \and Henry McNulty  }

\date{September 2025}

\begin{document}

\maketitle
\begin{abstract}
   In this short note, we study the impact of data augmentation on the smoothness of principal components of high-dimensional datasets. Using tools from quantum harmonic analysis, we show that eigenfunctions of operators corresponding to augmented data sets lie in the modulation space $M^1(\mathbb{R}^d)$, guaranteeing smoothness and continuity. Numerical examples on synthetic and audio data confirm the theoretical findings. While interesting in itself, the results   suggest that manifold learning and feature extraction algorithms can benefit from systematic and informed augmentation principles. 
\end{abstract}

\section{Introduction}

Expansions of functions in terms of basic building blocks, such as trigonometric functions or wavelets, are central in analysis and signal processing. In modern machine learning, features learned by deep networks often resemble such expansions, suggesting that data lie near low-dimensional structures within high-dimensional ambient spaces. 

Manifold learning formalizes this observation: data often concentrate on a low-dimensional manifold, allowing efficient representations. Data augmentation, a common practice in machine learning, can be understood as a way to smooth operators associated with empirical data, revealing intrinsic structure and stabilizing principal components.

This paper studies the smoothing effect of augmentation on data operators, showing how convolution with augmentation measures regularizes principal components and aids feature extraction. To do so, we leverage recent insights and methods from quantum harmonic analysis (QHA), cf.~\cite{Werner84,DoLuefNuSk24,dolusk24} as well as some important results on  embeddings of modulation spaces~\cite{Okoudjou2008,BeOk20} and the relation between modulation space membership of symbols and operator properties~\cite{GrHe03}.

\section{Preliminaries and Notation}

Let $\mathcal{H} = L^2(\mathbb{R}^d)$ be the Hilbert space of square-integrable functions. For $f,g \in \mathcal{H}$, the rank-one operator $f \otimes g$ is defined by
\[
(f \otimes g)(h) = \langle h, g\rangle f, \quad h \in \mathcal{H}.
\]

A bounded linear operator $T:\mathcal{H}\to\mathcal{H}$ is \emph{trace-class} if 
\[
\|T\|_{\mathcal{T}} = \mathrm{tr}(|T|) < \infty,
\]
and \emph{Hilbert-Schmidt} if
\[
\|T\|_{\mathrm{HS}}^2 = \mathrm{tr}(T^* T) < \infty.
\]

The \emph{trace} of a trace-class operator $T$ is defined by
\[
\mathrm{tr}(T) = \sum_{n} \langle T e_n, e_n \rangle,
\]
where $(e_n)$ is any orthonormal basis of $\mathcal{H}$; the sum converges absolutely and is independent of the choice of basis.

We also consider the Feichtinger algebra $M^1(\mathbb{R}^d)$, consisting of functions whose short-time Fourier transform is integrable. Operators with kernels in $M^1 (\mathbb{R}^{2d})$ have convenient convolution and Fourier transform properties used later.
\subsection{Time-Frequency Representations}
Let  $\pi(z)$ denote a  time-frequency shift defined by  $\pi(z)g(t) =e^{2\pi it\omega}g(t-x)$. 
For $\psi, \phi \in L^2(\mathbb{R}^d)$ we define:

\begin{itemize}
    \item \textbf{Short-time Fourier transform (STFT)} $V_\phi \psi$ of $\psi$ with window $\phi$:
    \[
    V_\phi \psi(z) = \langle \psi, \pi(z)\phi \rangle, \quad z \in \mathbb{R}^{2d},
    \]
    where $\pi(z)$ denotes the time-frequency shift.

    \item \textbf{Cross-Wigner distribution} of $\psi$ and $\phi$:
    \[
    W(\psi, \phi)(x, \omega) = \int_{\mathbb{R}^d} \psi\Big(x+\frac{t}{2}\Big) \overline{\phi\Big(x-\frac{t}{2}\Big)} e^{-2\pi i \omega \cdot t} \, dt.
    \]
\end{itemize}
\section{Operators, Translations and Fourier Transform}
Operators act on functions, respectively elements of  vector spaces. Thus they are potentially higher-dimensional than the objects they act on: they need to 
encode the latter's structure and the structure of the interaction between the isolated objects. We can understand functions as operators by looking at their action on other functions $f$ by means of projections: 
$$(g\otimes h)(f)=\langle f,h\rangle g .$$
If we think of the simple case of finite 
dimensional operators, which can be represented as finite matrices, we immediately see, that these simple projections form very special cases of operators. 
In many cases, though, the complicated action of an operator on functions
is again described with the help of more functions, e.g. the eigenfunctions in the case of a compact selfadjoint operator $T$ on $X$:
for some orthonormal set $\{g_n\}$ in $X$ we have 
\[ T=\sum_n s_n(T) g_n\otimes g_n. \]
Now one of the reasons of the importance
of Fourier transforms (and their relatives) in signal processing and other applications is the fact, that they diagonalize time-invariant systems. Systems, hence, which are invariant under translation. Can there be similar helpful connections for operators? 
What is the appropriate concept of translation, and of Fourier transform?

In order to understand the corresponding concepts for operators, first introduced by Werner~\cite{Werner84}, let us consider the properties of translation in the context of harmonic analysis. Indeed, Fourier transform turns translation into multiplication by a complex number with 
absolute value $1$. Also, the complex exponential functions $e_{\omega_0} (x)  = e^{2\pi i x \omega_0}$, for some fixed frequency $\omega_0$, are functions on $\R$. So, in the same spirit, we look for 
operators, whose inner products with 
operators turn translations of operators into multiplications with complex numbers. 

It turns out, that the following {\it translation of operators} provides the desired properties if paired with an equally appropriate Fourier transform for operators.
\subsection{Translation of Operators}

For $z \in \mathbb{R}^{2d}$ and $S$ a trace-class operator, define the operator translation
\[
\alpha_z(S) = \pi(z) S \pi(z)^*,
\]
where $\pi(z)$ is the  time-frequency shift defined by  $\pi(z)g(t) =e^{2\pi it\omega}g(t-x)$.

Since we are interested in defining a convolution for operators, which features properties similar to the classical convolution of functions in light of the Fourier transform, we must first find an equivalent for the integral of functions. 
Consider the defining properties of operator trace $\operatorname{tr}(T)$:   for two trace class operators $T, T'$, and r $c\in\C$ it holds that
\begin{itemize}
    \item $\operatorname{tr} (T+T') = \operatorname{tr}( T ) + \operatorname{tr}( T' )$
    \item $\operatorname{tr}( c T ) = c \operatorname{tr}( T )$
    \item $\operatorname{tr}( T T' ) = \operatorname{tr}( T T' )$. 
\end{itemize}
In  light of these connections, 
the following substitutions in the definition of convolutions, first proposed by Werner~\cite{Werner84}, seem intuitive: 
\begin{center}
\begin{tabular}{ |c|c| } 
 \hline
FUNCTIONS & OPERATORS \\
 \hline\hline
 real functions & selfadjoint operators\\
 \hline
 integral & trace  \\
 \hline
 translations & automorphism $\alpha_z$  \\ 
 \hline
 $L^p$-spaces & Schatten classes $\mathcal{S}^p$  \\ 
 \hline
continuous functions vanishing at infinity & compact operators  \\ 
 \hline
\end{tabular}
\end{center}

We will now introduce an analogue of the Fourier transform for a trace class operator $S$.
\subsection{Fourier-Wigner and Weyl Transform}
The Fourier-Wigner transform of a trace-class operator $S$ is
\[
\mathcal{F}_W(S)(z) = e^{-\pi i x \cdot \omega} \mathrm{tr}(\pi(-z) S), \quad z=(x,\omega) \in \mathbb{R}^{2d}.
\]
It satisfies
\[
\mathcal{F}_\sigma(S \star T) = \mathcal{F}_W(S) \mathcal{F}_W(T),
\]
linking operator convolution and Fourier analysis.\\

\begin{example}

Let $S = \varphi_2 \otimes \varphi_1$ with $\varphi_1, \varphi_2 \in L^2(\mathbb{R}^d)$.  
Then the Fourier-Wigner transform of $S$ is given by
\[
\mathcal{F}_W(\varphi_2 \otimes \varphi_1)(z) = A(\varphi_2, \varphi_1)(z),
\]
where $A(\varphi_2, \varphi_1)(z)$ is the \emph{cross-ambiguity function}.
In particular, for  the Gaussian 
\[
\varphi(t) = 2^{d/4} e^{-\pi t\cdot t}, \qquad t\in \mathbb{R}^d,
\] 
and the operator $S = \varphi \otimes \varphi$. Then its Fourier-Wigner transform satisfies
\[
\mathcal{F}_W\,S(z) = e^{\pi i x\cdot \omega} V_\varphi \varphi(z),
\]
where $V_\varphi \varphi$ is the short-time Fourier transform of $\varphi$ with itself.  
Explicitly, we find
\[
\mathcal{F}_W(\varphi \otimes \varphi)(z) = e^{2\pi i x\cdot \omega} \, e^{-\frac{\pi}{2} z\cdot z}, \qquad z\in \mathbb{R}^{2d}.
\]
\end{example}
\noindent Using the cross-Wigner distribution, we may introduce the Weyl calculus:\\
For $\sigma \in S_0(\mathbb{R}^{2d})$ and $\psi, \phi \in S(\mathbb{R}^d)$, the \emph{Weyl transform} $L_\sigma$ of $\sigma$ is the operator defined by
\[
\langle L_\sigma \psi, \phi \rangle_{S_0,S} = \langle \sigma, W(\psi, \phi) \rangle_{S_0,S}.
\]
Here, $\sigma$ is called the \emph{Weyl symbol} of the operator $L_\sigma$.\\
\begin{re}
Let  $\mathcal{F}_s$ denote the symplectic Fourier transform
\begin{align}\label{symfourierdef}
    \mathcal{F}_s f(z) := \int_{\mathbb{R}^{2d}} f(z')e^{-2\pi i\Theta (z,z')}\, dz'
\end{align}
where $\Theta$ is the symplectic form
\begin{align}\label{sympformdef}
    \Theta(z,z') := x'\omega-x\omega'.
\end{align}
Then the Fourier-Wigner transform and Weyl symbol of an operator $S$ are related by: 
$$\mathcal{F}_W(S)(z) = \mathcal{F}_s (S)  ( z).$$
\end{re}
 \subsection{Convolutions}  Convolution is an utterly important concept in signal processing. It is used in digital signal processing to study and design linear time-invariant systems such as digital filters. In  image processing, convolutional filtering can be used to implement algorithms such as edge detection, image sharpening, and image blurring, to name a few. 
\textbf{Convolution}  of two appropriately chosen functions $f$ and $g$ is defined by
\[ f\ast g(x)=\int_\R f(y)g(x-y)\,dy=\int_\R f(x-y)g(y)\,dy\]

Observe that  $f\ast g(x)=\int_\R f(y)T_y\check{g}(x)\,dy$, where $\check{g}(x)=g(-x)$ denotes the {\bf flip/parity } operator, also denoted by $Pg(x)=g(-x)$.  In other words, the convolution is an operator that integrates the translates of the function $f$ weighted by the function $g$: $f\ast g=\int_\R g(y)T_yf\,dy$

Given the beautiful properties and tremendous practical importance of convolution, the question of generalisation to operators arises. 
\subsubsection{Convolution of Operators}
For two operators $S,T$ on $L^2(\mathbb{R})$, their convolution $S\star T$ is the function on $\mathbb{R}^2$: 
    \[S\star T(z)=\mathrm{tr}(S\alpha_z(\check{T})),\]
				where $\check{T}=PTP$ is the flip of the operator $P$. 
On the other hand,  for a function $F$ on $\R^2$ and an operator $S$ we define their convolution $F\star S$ by 
     \[F\star S=\int_{\mathbb{R}^2} F(z) \alpha_z(S) \, dz.\]
The best-known existing example of such a convolution is given by classical localization operators. If we recall the properties of their spectral decomposition, we note a specific behaviour of the eigenvalues,~\cite{Daubechies1988,dolusk24}. 

 There is a very interesting and insightful connection between        
convolutions and Weyl quantization, namely: 
 $\alpha_z(L_a)=L_{T_za}$, i.e. translation of a Weyl operator $L_a$ with symbol $a$ is a Weyl operator with a translated symbol $T_za$.  

 Furthermore, as expected, the Fourier-Wigner transform shares several properties with the Fourier transform of functions:

\paragraph{Riemann-Lebesgue lemma.}  
If \(S \in \mathcal{T}^1\) (trace-class operators), then \(\mathcal{F}_W(S)\) is continuous and vanishes at infinity:
\[
\lim_{|z|\to\infty} |\mathcal{F}_W(S)(z)| = 0.
\]

\paragraph{Mapping properties.}  
Let \(f \in L^1(\mathbb{R}^{2d})\) and \(S,T \in \mathcal{T}^1\). Then
\begin{align}
\mathcal{F}_s(S \star T) &= \mathcal{F}_W(S)\, \mathcal{F}_W(T), \\
\mathcal{F}_W(f  \star S) &= \mathcal{F}_s(f)\, \mathcal{F}_W(S).
\end{align}
where \(\mathcal{F}_s\) denotes the symplectic Fourier transform. 
\section{Data Operators and Augmentation}
\subsection{Operators for  data sets and augmented data sets }

Given a dataset $\mathcal{D} = \{f_i\}_{i=1}^N \subset \mathcal{H}$, define the empirical data operator
\[
S_\mathcal{D} = \sum_{i=1}^N f_i \otimes f_i.
\]

Let $\Omega\subset \R^{2d}$, the augmented data set  is defined by 
 $$\D_\Omega =  \{(|\Omega|^{-1/2}\pi (\mu ) f_i;  f_i\in \D ,\mu\in\Omega\}$$ and the corresponding data operator is 
\[
S_{\mathcal{D}_\Omega} = \frac{1}{|\Omega|} \int_\Omega \alpha_z(S_\mathcal{D}) \, dz = \frac{1}{|\Omega|}\chi_\Omega \star S_\mathcal{D}.
\]
This is the mixed state localization operator corresponding to 
 $S$,  $\Omega$:
 \begin{equation}\label{eq:TF_augmentation} 
S_{\D_\Omega} = \frac{1}{|\Omega|}\chi_\Omega \star S =\frac{1}{|\Omega|}\sum_{\mu\in\Omega} \sum_i \pi (\mu ) f_i\otimes \pi (\mu ) f_i\
\end{equation}
\begin{re}
    The principal components of the augmented dataset correspond to the eigenfunctions of $S_{\mathcal{D}_\Omega}$. Convolution with $\chi_\Omega$ smoothes the operator, regularizing the spectrum and stabilizing feature extraction.
\end{re}
\begin{re}

Due to the structure of the convolutional layers in a Concolutional Neural Network (CNN), %as described in the introduction, 
the output of the first convolutional layer with input $F^0 = | V_g f|^2$  and convolutional kernels $m_k$, can be written as 
 \begin{align*}
    F^1 (z,k) &= (F^0\ast m_k )(z) \\
    &=\langle \int_{z'} V_g f(z' )\cdot m_k (z-z' ) \pi (z' ) g\, dz' , f\rangle\\
    &=\langle(T_z \check{m}_k\star g\otimes g) f, f\rangle = \check{m}_k\ast [(f\otimes f)\star (\check{g}\otimes \check{g})](z)\\
    &=[\check{m}_k\star (f\otimes f)]\star (\check{g}\otimes \check{g})(z)
\end{align*}
 As a consequence, we can define the operator $H_{\check{m}_k,g}$ stemming from CNN's first layer: \begin{align*}
\langle H_{\check{m}_k,g} f,f\rangle &=\langle (T_z \check{m}_k\star g\otimes g) f, f\rangle = \check{m}_k\ast [(f\otimes f)\star (\check{g}\otimes \check{g})]\\
    &=[\check{m}_k\star (f\otimes f)]\star (\check{g}\otimes \check{g}).\end{align*}
We now use associativity of the convolutions and omit the $g$.  We thus consider the operator 
    $ H_{m_k,f} = \check{m}_k\star (f\otimes f)$ instead.
 Finally, setting $ S = \sum_i f_i\otimes f_i$, and $m_k = \chi_\Omega$,  we obtain exactly the mixed-state localization operator with localization domain $\Omega$.        
\end{re}

%For a window $f \in \mathcal{H}$ and a mask $m \in L^1(\mathbb{R}^{2d})$, define
%\[
%H_{m,f} = \check{m} \star (f \otimes f),
%\]
%which generalizes classical localization operators. This formalism links augmentation and smoothing in a rigorous operator-theoretic framework.

\subsubsection{Impact of Augmentation in Manifold Learning}
Data augmentation plays a central role in manifold learning by densifying the sampling of high-dimensional datasets and stabilizing local neighborhood structures, which are crucial for algorithms such as Isomap or  Locally Linear Embedding.  When the available data provide only sparse observations from the underlying manifold, neighborhood graphs may become disconnected or distorted, leading to poor embeddings. Augmentation methods such as time warping, jittering, or scaling in time series data, create additional samples that remain close to the intrinsic geometry, thereby improving both the robustness and generalization of the learned manifold representation. Compared to classical linear methods such as Principal Component Analysis (PCA), which assumes global linearity of the data space, augmentation enables nonlinear manifold learners to better capture curved or folded low-dimensional structures by reinforcing invariances along the manifold. In essence, while PCA benefits from augmentation through increased variance coverage, nonlinear methods exploit augmentation to preserve local geometry and reveal the manifold’s intrinsic topology.
In most physical settings, we expect a continuous and often smooth time series.  
However, for a general Hilbert--Schmidt operator, we cannot expect the eigenfunctions, or principal components, to be smooth or even continuous.  

Here we present an advantage of data augmentation: appropriately augmenting a data set ensures that principal components are smooth and continuous.  
We will see that the conditions for such an augmentation are in fact quite unrestrictive.  

We quantify smoothness of a function by membership in a specific modulation space $M^{p,q}(\mathbb{R}^d)$, which imposes a restriction on the (local) Fourier decay of a function.

\subsection{Transition from functions to operators}

There are several ways to associate to  $\mathcal{D}=\{f_1,...,f_n\}$ an  operator:
\begin{itemize}
    \item Data operator by empirical covariance: Given  a data set $\mathcal{D}=\{f_1,...,f_N\}$ , we  construct the data operator $S_{\mathcal{D}}=\sum_{i=1}^N f_i\otimes f_i$. Then
    \begin{enumerate}
        \item The Fourier-Wigner transform of $S_{\mathcal{D}}$ is given as  $\sum_{i}V_{f_i} f_i (z) $
        \item The Weyl symbol  of $S_{\mathcal{D}}$ is given as $\sum_{i}W_{f_i} f_i (z) $
       \item The total correlation  of $S_{\mathcal{D}}$ is given as $\widetilde{S_{\mathcal{D}}} (z) = \sum_{i,j}|V_{f_i} f_j (z)|^2 $
       
    \end{enumerate}

    \item Data operator by orthonormal basis: Given  a data set $\mathcal{D}=\{f_1,...,f_N\}$ and an orthonormal basis $\mathcal{B}=\{e_1,...,e_N\}$, we  construct the data operator $C_{\mathcal{D}}=\sum_{i=1}^N f_i\otimes e_i$. Then
    \begin{enumerate}
        \item The Fourier-Wigner transform of $C_{\mathcal{D}}$ is given as $\sum_{i}V_{e_i} f_i (z) $ \item The Weyl symbol  of $C_{\mathcal{D}}$ is given as $\sum_{i}W_{e_i} f_i (z) $
       \item The total correlation  of $C_{\mathcal{D}}$ is given as
$\widetilde{S_{\mathcal{D}}} (z) = \sum_{i,j} V_{e_i} f_j (z) \cdot $
       
    \end{enumerate}
    Note that $C_{\mathcal{D}}^\ast\cdot C_{\mathcal{D}} = S_{\mathcal{D}}$.
\end{itemize}

\section{Eigenfunctions of an Augmented Data Set are Smooth}

\subsection{Two preparatory lemmas}

For the proof of our main result, we will need the following convolution relations between modulation spaces.
\begin{lem}[Convolution relation, \cite{BeOk20}]\label{lem:convolution}
Given $1 \leq p,q,r,s,t \leq \infty$ satisfying
\[
\frac{1}{p} + \frac{1}{q} = 1 + \frac{1}{r},
\qquad
\frac{1}{t} + \frac{1}{t'} = \frac{1}{s},
\]
the inclusion
\[
M^{p,t}(\mathbb{R}^d) * M^{q,t'}(\mathbb{R}^d) \hookrightarrow M^{r,s}(\mathbb{R}^d)
\]
is continuous, with bound
\[
\|f * g\|_{M^{r,s}} \lesssim \|f\|_{M^{p,t}} \, \|g\|_{M^{q,t'}}.
\]
\end{lem}

We will also use the following fact for modulation spaces:

\begin{lem}[\cite{Okoudjou2008}]\label{lem:compact}
Let $1 \leq p,q \leq \infty$. Then
\[
M^{p,q}_{\mathrm{comp}}(\mathbb{R}^d) 
= \big( \mathcal{F}L^q \big)_{\mathrm{comp}}(\mathbb{R}^d),
\]
where $M^{p,q}_{\mathrm{comp}}(\mathbb{R}^d)$ (resp.\ $(\mathcal{F}L^q)_{\mathrm{comp}}(\mathbb{R}^d)$) is the subspace consisting of functions with compact support.
\end{lem}

As a consequence of the last lemma, for any compact domain $\Omega$ we have $\chi_\Omega \in M^{1,2}(\mathbb{R}^{2d})$, since $\chi_\Omega \in L^2(\mathbb{R}^{2d})$.  

Invoking Lemma~\ref{lem:convolution} with $p = s = 1$ and  $t = t'= q = r  =2$, we thus have that \begin{align*}
        \chi_{\Omega} * F\in M^{2,1}(\mathbb{R}^{2d}).
    \end{align*}
    for any $F\in L^2(\mathbb{R}^{2d})$. 

   We move on to state and prove the main theorem.

  \subsection{Main result: The smoothing effect of augmentation}
In our central result, we  quantify smoothness of a function by membership in  $M^1(\mathbb{R}^d)$, which imposes a restriction on the (local) Fourier decay of a function. Note that more refined decay properties will be guaranteed by assuming membership in  specific, possibly weighted,  modulation spaces $M_m^{p,q}(\mathbb{R}^d)$.   

    \begin{thm} Given $S\in\mathcal{HS}$ and some compact domain $\Omega \subset \mathbb{R}^{2d}$, the eigenfunctions of $\chi_{\Omega} \star S$ are in $M^1(\mathbb{R}^d)$.
   
\end{thm}

\begin{proof}
We consider the Weyl symbol of the mixed-state localisation operator:
\[
\sigma_{\chi_\Omega * S} = \chi_\Omega * \sigma_S.
\]
By our assumptions on $S$ and $\Omega$, along with Lemma \ref{lem:compact}, we have
\[
\chi_\Omega * \sigma_S 
\in M^{1,2}(\mathbb{R}^{2d}) * M^{2,2}(\mathbb{R}^{2d}) 
\hookrightarrow M^{2,1}(\mathbb{R}^{2d}).
\]

In order to relate this observation to the properties of the augmented operator's eigenfunctions, we  use a particular instance of the  result in~\cite[Theorem 7.1]{GrHe03} on mapping properties between modulation spaces. While Gr\"ochenig and Heil gave general conditions on the Weyl symbol $\sigma_S$ of an operator $S$ in order for $S$ to be bounded from $M^{\mu_1, \mu_2}$ to  $M^{\nu_1, \nu_2}$, we only need the case $p=2$, $q=1$ and deduce that  $\chi_\Omega *S$ is bounded from $M^2(\mathbb{R}^d)$ to $M^1(\mathbb{R}^d)$, that is, $\nu_1 =  \nu_2 = 1 $ and  $\mu_1 =  \mu_2 = 2 $.

In order to conclude, we note that any function satisfying the eigenfunction equation
\[
\chi_\Omega * \sigma_S f = \lambda f
\]
for some constant $\lambda\in\C$ must therefore be in $M^1(\mathbb{R}^d)$, as claimed.
\end{proof}

\subsubsection{Numerical evidence}
We give two numerical examples in order to visually substantiate our results. 
\begin{example}[Gaussians]
The first experiment shows the smoothing effect on a synthetic example. We first consider the one-dimensional data set consisting of scaled versions of one Gaussian function $g$, yielding a rank-one data operator $S_0  = g\otimes g$. Note that its augmentation on $\Omega$ corresponds just to the classical localization operator $\chi_\Omega\star g\otimes g$, and thus the first eigenfunction of both $S_0$ and $\chi_\Omega\star g\otimes g$ is just a (modulated and shifted, depending on $\Omega$) Gaussian window, see the left plots in Figure~\ref{Fig1}. The (trivial) data manifold is then disturbed by small random time-frequency-shifts $\pi (z)$ with $|z|\leq 5$ to yield the data operator $S_5 =\sum_{z\in I} \pi(z) g\otimes\pi (z) g$.  The first eigenfunction of this operator is noisy, as seen in the upper right plot. The lower right plot shows the effect of augmentation: the eigenfunction is smoothed and seems to be close to the original data manifold. 
\begin{figure}[H]
        \centering
        \includegraphics[width=0.98\textwidth]{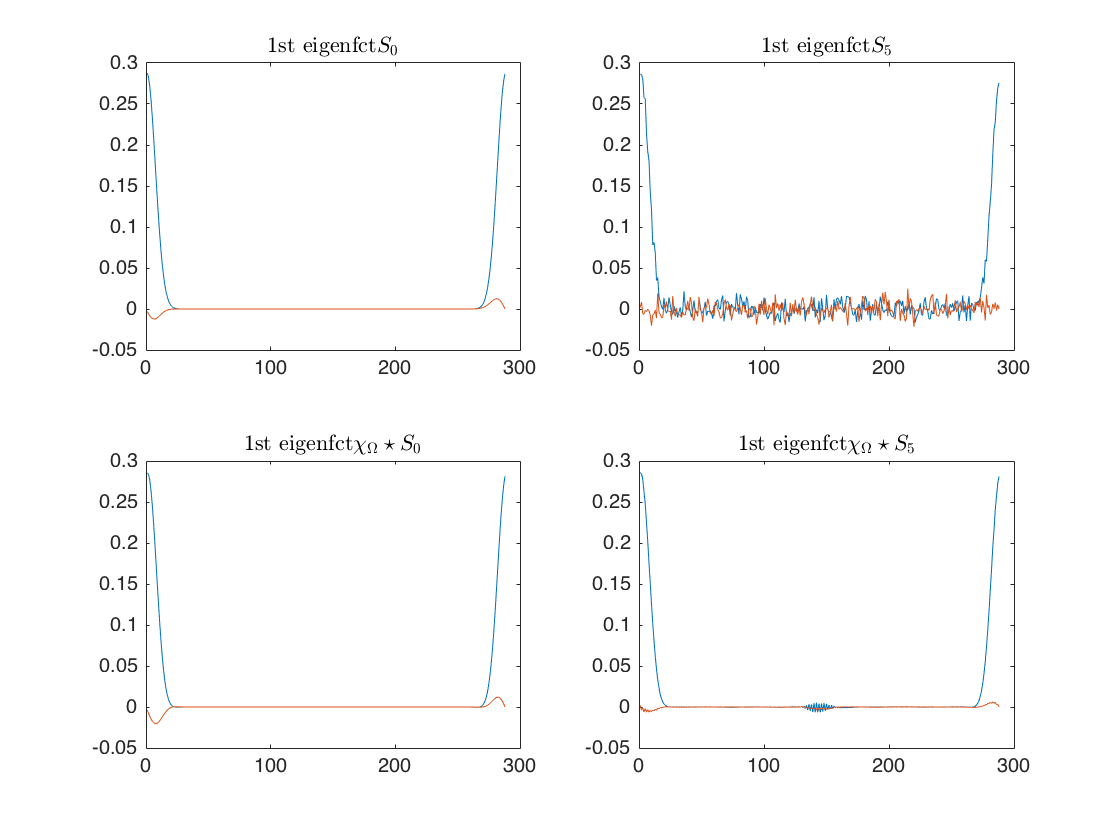}\vspace{-.3cm}
        \caption{First PC  of  data set and  TF-augmented data set}\label{Fig1}
    \end{figure}
\end{example}

\begin{example}[A music piece]
    As a second example, we chose a collection short audio-segments stemming from a music piece of approximately 7 minutes length. The music piece consists of different parts during which various instruments are audible or not. A possible future task will be the detection of the instruments playing at a particular moment in the piece. However, there is obviously a high random variability in the instances present in each time-snippet and the set of principal components is very noisy,  cf. the spectrograms shown in the two upper plots in Figure~\ref{Fig2}. The data set was then augmented along both time and frequency, that is, by time-frequency shifts in  a compact set $\Omega$ as in Example~1 and the resulting principal components show a significantly faster decaying spectrogram, which is in line with the main result of this contribution. 
\begin{figure}[H]
\centering
\includegraphics[width=0.95\textwidth]{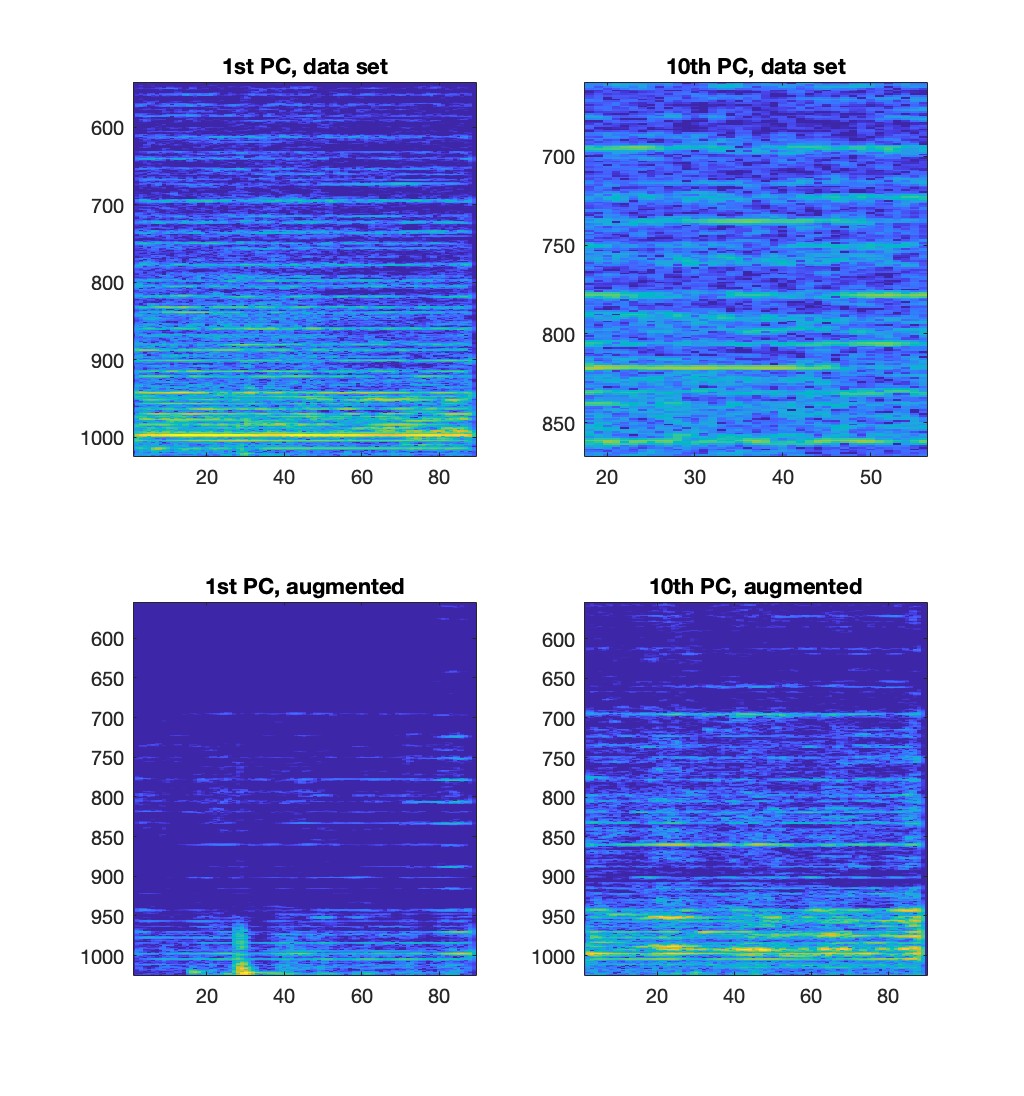}
  \caption{First and tenth principal   of  music data set and  TF-augmented music data set. The spectrograms are depicted and show the significantly smoother nature of the eigenfunctions after augmentation. }\label{Fig2}
\end{figure}
\end{example}
\section{Perspectives}
This note presented a fundamental result on the smoothing action of augmentation on data sets. 
In the next steps, these insights will be extended to more detailed decay conditions, in particular to mixed and weighted modulation spaces. On the other hand, the promising experimental results will be extended and leveraged in improvement for existing manifold learning strategies as outlined in the introduction. 

%\bibliography{AugmBib}
%\bibliographystyle{abbrv}

%\Addresses

\end{document}